\documentclass[12 pt]{article}
\usepackage{DEF}
\usepackage{DEF-script}

\def\lln{\ell_n}
\def\llnm{\ell_n^m}

\def\lan{\bar\ell_n}
\def\lanm{\bar\ell_n^m}

\def\rem{r^m}
\def\renm{r_n^m}
\def\Z{C}

\def\mX{\mathcal{X}}
\def\mN{\mathcal{N}}
\def\mY{\mathcal{Y}}
\def\Ex{\mathbb{E}}
\def\VAR{{\rm VAR}}
\def\Rl{\mathbb{R}}

\def\th{\theta}
\def\tht{\th_\star}
\def\thn{\hat\th_n}
\def\thnm{\hat\th_n^m}

\def\Zet{{Z}}
\def\gie{{\phi}}

\def\grad{\nabla}
\def\hess{\nabla^2}

\def\app{_{\rm approx.}}

\def\Dnm{D_n^m}

\def\topr{\to_{\rm p}}
\def\tod{\to_{\rm d}}

\def\op{{o_{\rm p}}}

\title{Asymptotics of maximum likelihood estimators based on Markov chain Monte Carlo methods}

\author{
B{\l}a\.{z}ej Miasojedow
\footnote{Institute of Applied Mathematics and Mechanics,
University of Warsaw, Banacha 2, 02-097 Warsaw, Poland, {\tt  	B.Miasojedow@mimuw.edu.pl}
},
Wojciech Niemiro
\footnote{Institute of Applied Mathematics and Mechanics,
University of Warsaw, Banacha 2, 02-097 Warsaw, Poland, {\tt wniem@mimuw.edu.pl}
and Faculty of Mathematics and Computer Science, Nicolaus Copernicus University,  Chopina 12/18, 87-100, Toru\'n, Poland {\tt wniem@mat.umk.torun.pl}
}
and Wojciech Rejchel
\footnote{corresponding author, Faculty of Mathematics and Computer Science,
Nicolaus Copernicus University,  Chopina 12/18, 87-100, Toru\'n, Poland, +48 566112943,
{\tt wrejchel@gmail.com}}
}

\date{}
\begin{document}

\maketitle       
   
\begin{abstract}
In many complex statistical models maximum likelihood estimators cannot be calculated. 
In the paper we solve this problem using Markov chain Monte Carlo approximation of the true likelihood. 
In the main result we prove asymptotic normality of the estimator, when both sample sizes 
(the initial and Monte Carlo one) tend to infinity. Our result can be applied to models with intractable norming constants
and missing data models. 
\end{abstract}                           


\section{Introduction}\label{Sec:Introduction}

Maximum likelihood (ML) is a well-known and often used method in estimation
of parameters in statistical models. However, for many complex models exact calculation of ML estimators is very difficult or
impossible. Such problems arise in missing data models or if considered densities are known only up to intractable
norming constants, for instance in Markov random fields or spatial statistics. 
In missing data models many Monte Carlo (MC) or Markov chain Monte Carlo (MCMC) methods have been proposed to approximate the observed likelikelihood 
\cite{GelfandCarlin1993, Geyer1994,Kongetal1994}. There are also Monte Carlo methods for maximum likelihood that do not approximate the likelihood 
\cite{Penttinen1984, WeiTanner1990, Younes1988} and non-Monte Carlo methods dedicated to this problem \cite{DempsterLairdRubin1977}. 
The literature dedicated to the problem  of intractable norming constant is extensive as well. Among the proposed methods
we should  mention the maximum pseudolikelihood \cite{Besag1974}, 
Monte Carlo maximum likelihood (MCML) \cite{Cappe2002, Geyer1994} and Markov chain Monte Carlo maximum likelihood (MCMCML)
\cite{GeyerThom1992, Geyer1994}. In the current paper we focus on the MCMCML method.

In influential papers \cite{GeyerThom1992, Geyer1994} the Authors prove consistency and asymptotic normality of  MCML estimators 
under the assumption that the initial sample is fixed, and only the Monte Carlo sample size tends to infinity.
Both sources of randomness (one due to the initial sample and the other due to Monte Carlo simulations) are considered in
\cite{Cappe2002, MNPR2016a, GeyerSung2007,  Imput2010}. The authors of the first mentioned paper apply the general importance sampling recipe. 
They show that for their scheme of simulations,
the Monte Carlo sample size has to grow exponentially fast to ensure consistency of the estimator.
As the remedy for this problem they propose to use a preliminary estimator which is consistent. Another
possibility to overcome this problem is proposed in \cite{GeyerSung2007}.
The log-likelihood is first decomposed into independent summands and then importance sampling is applied.
Papers \cite{Cappe2002, GeyerSung2007}   describe asymptotic properties of MCML estimators only for models with missing data while 
\cite{MNPR2016a} investigates models with intractable norming constants and explanatory variables.
However, in \cite{Cappe2002, MNPR2016a, GeyerSung2007} 
it is assumed  that one can efficiently generate independent samples from a given distribution. 
In many practical problems it is impossible and instead of independent sampling one has to use Markov chain simulation. 
The goal of the current paper is to investigate asymptotic properties of estimators obtained in this way, i.e. MCMCML. 
The MCMCML approach  has been successfully applied in practice 
\cite{GeyerThom1992, HuWu1998, WuHu1997, Imput2010}. 
However, to the best of our knowledge, there is no full theoretical justification of it in the literature. 
Our paper fills this gap and can be viewed as a generalization of the results contained in \cite{MNPR2016a, GeyerSung2007}.
The methods used in these papers are extended and developed to work in the case where the MC sample is a Markov chain. 

The main result of the current paper is the asymptotic normality of MCMCML estimators.
We focus on models with intractable norming constants but the main result can be directly applied also to models with  missing data. 
We prove our theorem using classical methods from the empirical process theory, but our argumentation is not standard, 
because we consider two sources of randomness (due to the initial sample and the MCMC sample). 
Moreover, the MCMC approximation, given in \eqref{eq:fr} below, is a sum of two expressions with a rather complicated
dependence structure (the second sum is conditionally a functional of a Markov chain but also depends on the initial sample).
Even though we work in a more difficult scenario, we significantly simplify argumentation and weaken regularity assumptions comparing to \cite{GeyerSung2007}.
It is discussed in detail in Section \ref{sec:main_result}.  

The paper is organized as follows: in Section \ref {sec:model_desc} we describe the models under consideration. 
In Section \ref{sec:main_result} we state the main result (Theorem 
\ref{with_cov}) and show its applications. The proof of this theorem is given in Section \ref{sec:proof}.

\section{Description of the models}
\label{sec:model_desc}

In the paper we consider models with intractable norming constants and  models with missing data. We 
focus on the former but the latter can be investigated similarly.

\subsection{Models with intractable constants}
\label{subsec:intr_constant}
 
We consider the following parametric model with covariates
\begin{equation}\nonumber
p(y|x,\th)=\frac{1}{\Z(x,\th)}f(y|x,\th),
\end{equation}
where $y \in \mY \subset \mathbb{R}^d$ is a response variable, $x \in \mX \subset
\mathbb{R}^l$ is a covariate or ``explanatory'' variable (random or deterministic), $\th \in \mathbb{R} ^p$ is a parameter
describing the relation between $y$ and $x$. The norming constant,
\begin{equation}\nonumber
\Z(x,\th) = \int f(y|x,\th) dy,
\end{equation}
is difficult or intractable.

We assume that the data consist of $n$ independent observations $(Y_1,X_1),\ldots,$ $(Y_n,X_n).$
If we regard covariates as random, then we assume that these pairs form an i.i.d.\  sample from a joint distribution
with a density $g(y,x)$. Alternatively, $x_i$ can be regarded as deterministic and then we assume that random variable
$Y_i$ has a probability distribution $g_i$ which depends on $x_i$. Both cases can be analysed very similarly.
For simplicity we focus attention on the model with random covariates.
It is not  necessary to assume that $g(y | x)=p(y|x,\th_0)$ for some $\th_0$. The case when no such $\th_0$ exists, i.e.\
the model is misspecified, makes the considerations only slightly more difficult. Thus,
let us consider the following log-likelihood
\begin{eqnarray}\nonumber \label{eq:likn}
 \lln(\th) &=&\log p(Y_1,\ldots,Y_n|X_1,\ldots,X_n,\th)\\
           &=&\sum_{i=1}^{n} \log  f(Y_i|X_i,\th)- \sum_{i=1}^{n} \log {\Z(X_i,\th)}  .
\end{eqnarray}
The first term in (\ref{eq:likn}) is easy to compute while the second
one is approximated by Markov chain Monte Carlo. Let $h(y)$ be an importance sampling (instrumental)
distribution and  note that
\begin{equation}\nonumber
{\Z(x,\th)}=\int f(y|x,\th) d y=\int \frac{f(y|x,\th)}{h(y)}h(y) dy=\Ex_{Y\sim h}\frac{f(Y|x,\th)}{h(Y)}
\end{equation}
for fixed $\th, x.$
Therefore, a natural approximation of the norming constant is
\begin{equation}
\label{Cm}
C_m(x,\th)=\frac{1}{m}\sum_{k=1}^{m}\frac{f(Y^k|x,\th)}{h(Y^k)} \:,
\end{equation}
where $Y^1,\ldots, Y^m$ is a  sample drawn from $h$ or, which is more realistic and is considered in the current paper,
$Y^1,\ldots, Y^m$ is a Markov chain with $h$ being a density of its stationary distribution. The MCMC sample is independent of the initial sample. From LLN for Markov chains  we have 
$
C_m(x,\th)
\rightarrow C(\th),
$ when $m \rightarrow \infty$ and $\th,x$ are fixed.
Thus, an MCMC approximation of the log-likelihood $\lln(\th)$ is
\begin{equation} \label{eq:fr}
\llnm(\th) =\sum_{i=1}^{n }\log f(Y_i|X_i, \th)-\sum_{i=1}^{n}\log C_m(X_i, \th) ,
\end{equation}
and its maximizer is denoted by $\thnm .$

Let us note that the general Monte Carlo recipe 
can also lead to approximation schemes different from \eqref{eq:fr}. For instance, we could generate $n$ independent
MCMC samples instead of one, i.e.\  $Y_i^1,\ldots,Y_i^m \sim h_i, i=1,\ldots,n$ and use
the $i$-th sample to approximate $\Z (x_i, \theta).$ Using this scenario one can obtain
estimators with better convergence rates, but at the cost of increased  computational
complexity. Another scheme, proposed in \cite{Cappe2002}, approximates the log-likelihood by
\begin{equation}\nonumber
\sum_{i=1}^{n }\log f(Y_i|X_i, \th)-\log \frac{1}{m}\sum_{k=1}^{m}\prod_{i=1}^{n}\frac{f(Y_i^k|X_i, \th)}{h_i(Y_i^k)}. \end{equation}
{However, this scheme leads to estimators with unsatisfactory asymptotics unless a preliminary estimator is used.
Thus, we focus our attention only on $(\ref{eq:fr}).$

\subsection{Models with missing data}
\label{subsec:miss_model}

These models are the same as considered e.g.\ in \cite{GeyerSung2007} but our notation is slightly different. 
We assume that $x$ is observed and $y$ is missing in the complete data $(x,y).$ 

The joint density is denoted by $f(x,y|\th)$ while the unavailable marginal density is $f(x|\th)=\int f(x,y|\th) \d y.$ 
Let the observed data $X_1, \ldots, X_n$ be i.i.d.\ from some density $g.$ Obviously, a maximizer of the log-likelihood
$$
\sum_{i=1}^n \log f(X_i|\th)
$$
cannot be calculated. We use a Markov chain $Y^1, \ldots, Y^m$ with the stationary distribution $h,$ which is independent of $X_1, \ldots, X_n,$  
to approximate the unknown marginal density $f(x|\th)$ by 
$$
\frac{1}{m}\sum_{k=1}^{m}\frac{f(x,Y^k|\th)}{h(Y^k)} \:.
$$
Therefore, the MCMCML estimator is a maximizer of 
$$
\sum_{i=1}^n \log \left[\frac{1}{m}\sum_{k=1}^{m}\frac{f(X_i,Y^k|\th)}{h(Y^k)}\right]\:,
$$
which is equivalent to 
\begin{equation}
\label{approx_miss}
\sum_{i=1}^n \log f (X_i|\th) + \sum_{i=1}^n \log  \frac{1}{m}\sum_{k=1}^{m}\frac{f(Y^k|X_i,\th)}{h(Y^k)}  \:.
\end{equation}

\section{Main result}
\label{sec:main_result}

In this section we state the key theorem of the paper and describe its applications.
We focus on models with intractable norming constants. A similar result for missing data models 
is only briefly commented on, because it is a straightforward modification of our main theorem.

We  need the following notations: the MCMC approximation (\ref{eq:fr}) multiplied by $\frac{1}{n}$ is denoted by $\lanm(\th)$ and decomposed as follows
\begin{equation} \label{eq:likm}
\lanm(\th) =\lan(\th) - \rem_n(\th),
\end{equation}
where
\begin{eqnarray*}
\lan(\th)  &=& \frac{1}{n} \sum_{i=1}^{n}\left[ \log  f(Y_i|X_i,\th)-\log {\Z(X_i,\th)}\right] = \frac{1}{n} \sum_{i=1}^{n} \log  p(Y_i|X_i,\th),  \\
\rem_n(\th) &=& \frac{1}{n} \sum_{i=1}^n \left[ \log  C_m (X_i, \th) -
\log {\Z(X_i,\th)} \right]= 
\frac{1}{n} \sum_{i=1}^n  \log  \Zet _m (X_i, \th)
\;,
\end{eqnarray*}
where for fixed $\th,x$ 
$$
\Zet _m (x,\th) = \frac{1}{m} \sum_{k=1}^m
\frac{p(Y^k|x,\th)}{ h(Y^k)}\:.
$$
Let $\tht$ be a  maximizer of $\Ex_{(Y,X)\sim g} \log p(Y|X,\th)$, i.e.\ the
 Kullback-Leibler projection.
Finally, let symbols $\grad$ and $\hess$ denote derivatives with respect to $\th$ and introduce the following notations:
\begin{eqnarray*}
\Psi(y|x)& =& \frac{\grad p(y|x,\tht)}{h(y)} \:,\\
\bar{\Psi} (y)&=& \Ex_{X \sim g} \Psi(y|X). 
\end{eqnarray*}
Now we can state the main result for the model with intractable constants and covariates.

\begin{thm}\label{with_cov}
Suppose the following assumptions are satisfied:
\begin{enumerate}
\item $(Y^k)_{k \geq 1}$ is a  reversible and geometrically ergodic homogeneous Markov chain on the state space $\mY$ which has stationary distribution with a density $h$
 and initial distribution with density $q$ such that $\left\Vert q/{h}\right\Vert_\infty =\sup_y |q(y)/h(y)| < \infty $.  
\item second partial derivatives of $p(y |x,\th)$ with respect to $\th$ exist and are continuous for all $y$ and $x$, and may be passed under the integral sign in
$\int   p(y|x, \th) dy$ for fixed $x$,
\item $\thnm$ is a consistent estimator of $\tht,$
\item $\Ex _{(Y,X) \sim g } \nabla \log p(Y|X, \tht)=0, $ matrices
\begin{eqnarray*}
V&=&\VAR_{(Y,X)\sim g}\grad \log p(Y|X,\tht), \\
D&=&
\Ex_{(Y,X)\sim g} \hess  \log p(Y|X,\tht)
\end{eqnarray*}
are well defined and matrix $D$ is negative definite,
\item the expectation $\Ex_{Y \sim h, X\sim g}|\Psi(Y|X)|^2$ is finite, 
\item $\sup\limits_{\th\in U} | \hess \lan (\th) - \Ex_{(Y,X)\sim g} \hess  \log p(Y|X,\theta) | \topr 0 , \quad n
\rightarrow \infty,$  where  $U= \{\th : |\th - \tht| \leq \delta\} $ is some
neighbourhood of $\tht$ ($\delta>0$).  
\item \begin{enumerate}
\item $\sup\limits_{x \in \mX} |\Zet _m (x,\tht) - 1|  \topr 0, \quad
m \rightarrow \infty,$
 \item   $\sup\limits_{x \in \mX} |\nabla \Zet _m (x,\tht)| \topr 0, \quad
m \rightarrow \infty,$
\item  $\sup\limits_{\th \in U ,x \in \mX} |\hess \Zet _m (x,\th) | \topr 0, \quad
m \rightarrow \infty,$ for some neighbourhood $U$ of $\tht$.
\end{enumerate}
\end{enumerate}
Then 
\begin{equation}\nonumber
\left( \frac{V}{n} + \frac{W}{m} \right)^{- \frac{1}{2}}\: D \left(\thnm -
\tht
\right) \tod \mathcal{N} (0,I), \quad n,m \rightarrow \infty,
 \end{equation}
where 
\begin{equation}\nonumber
W=\VAR_{h}  \bar\Psi(Y^1) +2 \sum_{k=2}^\infty {\rm COV}_h \left(\bar\Psi(Y^1), \bar\Psi(Y^k)\right),
\end{equation}
where $\VAR_{h}$ and ${\rm COV}_h$ denote the stationary covariance matrices.
\end{thm}

In Theorem \ref{with_cov} we  prove that the  maximizer of (\ref{eq:fr}) satisfies
\begin{equation}\label{asymp}
\thnm
             \sim\app \mathcal{N} \left(\tht,D^{-1}\left(\frac{V}{n}+\frac{W}{m}\right)D^{-1}\right).
\end{equation}
Formula (\ref{asymp}) means that the estimator
$\thnm$ behaves like a normal vector with the mean $\tht$ when both the initial sample size $n$ and the
Monte Carlo sample size $m$ are large. Suppose that $\thn$ is a maximizer of $\lln(\th),$ that is a genuine maximum likelihood estimator, then
\begin{equation}\nonumber 
\thn \sim\app \mathcal{N} \left(\tht,\frac{1}{n}D^{-1}VD^{-1}\right),
\end{equation} 

Note that the first component of
the asymptotic variance in (\ref{asymp})  is the same as the asymptotic variance
of the maximum likelihood  estimator $\thn.$ The second component, $D^{-1}WD^{-1}/m$, is due to Monte Carlo randomness.
Furthermore, if $m$ is large, then asymptotic behaviour of $\thnm$ and $\thn$ is similar.
Finally, if the model is correctly specified, that is $g(y|x) = p(y|x , \th_0)$ for some $\th_0,$
then $\tht = \th_0$ and $D=-V.$

Now we discuss the assumptions of Theorem \ref{with_cov}. Assumption 1 relates to the MCMC sample and is not restrictive.
Conditions 2-5 are standard regularity assumptions. 
The key conditions in the theorem are Assumptions 6 and 7. They  stipulate uniform convergence of 
a sum of independent random variables (Assumption 6) and a sum along a trajectory of Markov chain (Assumption 7). We show that  these conditions  are satisfied in the widely-used autologistic model.

\begin{exam}
\label{Ex}
Consider the autologistic model with covariates defined as follows.
The response variable is a binary vector $y=(y(s),s=1,\ldots,d)\in\{0,1\}^d.$ The vector of covariates or ``explanatory'' variables is 
$x=(x(s),s=1,\ldots,l)\in \mathcal{X},$ where $\mathcal{X}$ is  a compact subset of $\Rl^{l}$. The parameter is 
$\th =(\alpha, \beta),$ where $\beta=(\beta_{r,s})$ and  $\alpha=(\alpha_{r,s})$ are matrices of dimensions 
$d\times d$ and $d\times l$, respectively. For identifiability, assume $\beta_{r,s}=\beta_{s,r}$.
The probability distribution on $\Y=\{0,1\}^d$ is defined as follows 
\begin{equation}\notag 
p(y|x,\beta,\alpha):=\frac{1}{\Z(x,\beta,\alpha)}\,\exp\left(
  \underbrace{\sum_{r=1}^{d}\sum_{s=1}^{d} \beta_{r,s}y{(r)}y{(s)}}_{\rm auto-regression} + 
  \underbrace{\sum_{r=1}^{d}\sum_{s=1}^{l}\alpha_{r,s}y{(r)}x{(s)}}_{\rm regression}\right).
\end{equation}
Obviously, the norming constant $\Z(x,\beta,\alpha)$ is  intractable for large $d$. 
We denote
$$
T(x,y) = \left(y(1)y(1),y(1)y(2), \ldots, y(d)y(d), y(1)x(1), \ldots, y(d)x(l) \right). 
$$
Then considering matrices $\alpha $ and $\beta$ as vectors we can simply write
$$
p(y|x,\theta) = \frac{\exp \left( \theta ' T(x,y)\right)}{C(x,\theta)} \:.
$$
Let $(Y^k)_{k \geq 1}$ be a Gibbs sampler on $\mathcal{Y}$ with the stationary density $h.$
Assumptions 1, 2, 4 and 5 of Theorem \ref{with_cov} are standard, therefore we focus on the remaining conditions. 
Notice that for fixed $\theta$ we have 
\begin{equation}
\label{exam1}
\lanm (\theta) \topr \Ex _{(Y,X) \sim g} \log p(Y|X,\theta) \text{ as } n,m \rightarrow \infty.
\end{equation}
Indeed, using assumption 6(a) we can easily prove that $r_n^m(\theta) \topr 0.$ Moreover, we can express
$$
\lanm(\theta) = \frac{1}{n} \sum_{i=1}^n \theta ' T(X_i, Y_i) -\frac{1}{n} \sum_{i=1}^n
\log \left[
\frac{1}{m} \sum_{k=1}^m \frac{\exp \left( \theta ' T(X_i,Y^k)\right)}{h(Y^k)}
\right] .
$$
The Hessian of $\lanm(\theta)$ is a weighted covariance matrix with positive weights that sum up to 1, so $\lanm$ is concave. 
This property and \eqref{exam1} implies convergence of maximizers in Condition 3. 
Next, we focus on Condition 6. Notice that 
$$
\hess \log p(y|x,\th)= -\hess \log \Z (x,\th)= -\frac{\hess \Z (x,\th)}{\Z(x,\th)}
+ \frac{\grad \Z (x,\th) \grad ^T \Z (x,\th)}{\Z ^2(x,\th)}\:,
$$
so this function is continuous in $x.$ Therefore, uniform convergence over the  set $U$ in this assumption is implied by 
\cite[Theorem 16(a)]{Ferguson1996} applied to the initial sample which involves i.i.d.\ random variables. Uniform convergence in  Conditions 7 (a)-(c)
relates to the MC sample which is a Markov chain. 
Notice that \cite[Theorem 16(a)]{Ferguson1996} can be extended to Markov chains, if  we apply SLLN 
for Markov chains on the top of \cite[page 110]{Ferguson1996}
 (in fact, this theorem can be extended to arbitrary sequence of random variables that admits 
SLLN). Basing on this argumentation we can prove that  Conditions 7 (a)-(c) holds using compactness of $\mathcal{X}, U$ and continuity of the function 
$T$ in $x.$ Note that  in Conditions 7(b) and 7(c) we have vector- and matrix-valued functions, respectively, but it is enough to  prove uniform 
convergence for each component. 
\end{exam}

Theorem \ref{with_cov} can be directly applied also to models with missing data as described in Section \ref{sec:model_desc}.
Indeed, using the close relation between MCMC approximations \eqref{eq:likm} and \eqref{approx_miss} we should just follow the proof of Theorem \ref{with_cov}.
Thus, under analogous regularity assumptions to Theorem \ref{with_cov} we obtain 
\begin{equation}\nonumber
\left( \frac{V}{n} + \frac{W}{m} \right)^{- \frac{1}{2}}\: D \left(\thnm -
\tht
\right) \tod \mathcal{N} (0,I), \quad n,m \rightarrow \infty,
 \end{equation}
where $V=\VAR_{X\sim g}\grad \log f(X|{\tht}),$ $D=\Ex_{X\sim g} \hess  \log f(X|{\tht})$
and  
\begin{equation}\nonumber
W=\VAR_{h} {\bar\Psi}(Y^1)  +2 \sum_{k=2}^\infty {\rm COV}_h \left({\bar\Psi}(Y^1), {\bar\Psi}(Y^k)\right)
\end{equation}
for $\bar\Psi (y)= \Ex_{X \sim g} {\nabla f (y|X,{\tht})}/{h(y)}\:.$
Therefore, we extend \cite[Theorem 2.3]{GeyerSung2007} from the i.i.d. case to the Markov chain case. The only price we pay for this generalization is 
(mild) Condition 1. Moreover,  
we  weaken their condition (6). Namely, we replace the requirement that the class is Donsker by the  requirement that the class is Glivenko-Cantelli.
This fact follows from  argumentation used to obtain
\begin{equation}\nonumber
\left( \frac{V}{n} + \frac{W}{m} \right)^{- \frac{1}{2}} \grad \lanm(0)
 \rightarrow_d \mathcal{N} (0,I), \quad n,m \rightarrow \infty
\end{equation}
that is needed in the proof of Theorem \ref{with_cov} and its analog in the proof of \cite[Theorem 2.3]{GeyerSung2007}. 
While \cite{GeyerSung2007} uses an arduous and complicated method based on weak convergence of stochastic processes and its properties 
(see \cite[Lemma A.4]{GeyerSung2007}), we show in Section \ref{sec:proof} that this analysis can be based on much simpler methods.

\section{Proofs}
\label{sec:proof} 

The following lemma about Markov chains plays the key role in the proof of Theorem {\ref{with_cov}}.

\begin{lem}
\label{aux_lemma}
Let $(Y^k)_{k \geq 1}$ be a reversible and geometrically ergodic homogeneous Markov chain with spectral gap $1-\rho$. Assume its stationary distribution $\pi$
 and initial distribution $\nu$  are such that $\left\Vert {\d\nu}/{\d\pi}\right\Vert_\infty =\sup_y |(\d \nu/\d \pi)(y)| < \infty $.  
 Let $\gie$ be a function such that $\int \gie(y)\pi(\d y) =0$ and $\int \gie(y)^2\pi(\d y) =\Vert \gie\Vert_\pi^2<\infty$. For all $k\leq l$ it holds that 
 $$\left|\mathbb{E}[ \gie(Y^k)\gie(Y^l)]\right|\leq \left\Vert \frac{\d\nu}{\d\pi}\right\Vert_\infty 
\Vert \gie\Vert_\pi^2 \, \rho^{l-k}.
$$
\end{lem}

Note that we will apply Lemma \ref{aux_lemma} in a situation where both the distributions $\pi$ and $\nu$ have densities $h$ and $q$, respectively.
Then the Radon-Nikodym derivative $({\d\nu}/{\d\pi})(y)$ is simply $q(y)/h(y)$.
First we give a proof of the main result, then a proof of Lemma \ref{aux_lemma}.

\begin{proof}[Proof of Theorem {\ref{with_cov}}]
Let us first describe the outline of the proof. 
The beginning of the proof is standard, namely we define
$$
\Dnm=\int_0^1 \hess \lanm \left( \tht + s(\thnm -\tht)
\right) ds.
$$
By Taylor expansion and the fact that $\thnm$ maximizes $\lanm$ we obtain
$$
-\grad \lanm (\tht) = \Dnm (\thnm-\tht).
$$
If we  show that 
\begin{equation}
\label{formD}
\Dnm \topr D, \quad n,m \rightarrow \infty
\end{equation}
 and 
\begin{equation}
\label{AS1}
\left( \frac{V}{n} + \frac{W}{m} \right)^{- \frac{1}{2}} \grad \lanm(\tht)
 \tod \mathcal{N} (0,I), \quad n,m \rightarrow \infty,
\end{equation}
then the conclusion of the theorem will follow from \eqref{formD}, \eqref{AS1} and  Slutsky's theorem. 

The argumentation to obtain \eqref{formD} is the same as in the proof of 
\cite[Lemma A.6]{GeyerSung2007}. The fact that the MC sample is a Markov chain instead of independent random variables  does not play any role in it. 
Therefore, we omit the proof of  \eqref{formD} and focus on \eqref{AS1}.

To show \eqref{AS1} we express $\grad \lanm (\tht)$ as follows:
 \begin{eqnarray}
 \grad \lanm(\tht) &=&  \grad \lan(\tht)-\grad \renm(\tht) 
                = \grad \lan(\tht)-\frac{1}{n}\sum_{i=1}^n \grad \log \Zet_m(X_i,\tht) \nonumber \\ 
&=& \grad \lan(\tht) \nonumber  \\ 
&-& \frac{1}{n}\sum_{i=1}^{n }\left[\frac{\nabla \Zet _m (X_i,\tht)}
                 {\Zet _m(X_i,\tht)}- \nabla \Zet _m (X_i,\tht) \right]\label{Term2}\\
&-&  \frac{1}{n}\sum_{i=1}^{n }\left[\grad \Zet _m (X_i,\tht)- \bar\Zet_m(\tht)\right] \label{Term3} \\ 
&-&   \bar\Zet_m(\tht), \nonumber
\end{eqnarray}
where  $\bar\Zet_m(\th)=\Ex_{X\sim g} \grad  \Zet _m (X,\th).$
The first term of the above displayed equation depends only on the initial, i.i.d.\ sample and $\sqrt{n} \nabla \lan(\tht)\tod \mathcal{N}(0,V)$
as $n\to \infty$.
This fact is well-known and follows from the CLT for i.i.d.\ variables. The last term  depends only on the
MCMC sample. Since
$$
 \bar\Zet_m(\tht)=\frac{1}{m}\sum_{1=1}^{m}\bar \Psi(Y^k),
$$
we can apply the CLT for Markov chains to infer that $\sqrt{m}   \bar\Zet_m(\tht)\tod \mathcal{N}(0,W)$
as $m\to \infty$. We will later prove that both the middle terms, \eqref{Term2} and \eqref{Term3} are 
negligible in the sense that they are $\op(1/\sqrt{m})$. Provided this is done, the rest of the proof is easy. 
We first assume that $\frac{n}{n+m}\to a$ and consider three cases corresponding to rates at which $n$ and $m$ go to infinity:
$0<a<1$,  $a=0$ and $a=1$. Once \eqref{AS1} is proved in these three special cases, the subsequence principle shows that it is valid in general 
(for $n\to \infty$ and $m\to \infty$ at arbirary rates). We consider only the case $0<a<1,$ because argumentation for the others is similar.
Since the Monte Carlo sample is independent of the observed one, we infer that
\begin{eqnarray}
&&        \sqrt{n+m} \grad \lanm(\tht) \nonumber\\
&=& \sqrt{\frac{n+m}{n}} \: \sqrt{n} \grad \lan(\tht)
- \sqrt{\frac{n+m}{m}} \: \sqrt{m}  \bar\Zet_m(\tht)+\op(1) \nonumber\\
   &\tod& \mN(0, V/a + W/(1-a)). \nonumber
\end{eqnarray}
 To finish the proof, just note that
$$\sqrt{n+m} \left(V/a + W/ (1-a)\right)^{-\frac{1}{2}}
\left(V/n + W/ m \right)^{\frac{1}{2}} \rightarrow I \quad n,m \rightarrow \infty.$$

We are left with the task of bounding the terms \eqref{Term2} and \eqref{Term3}. 
This is the difficult and novel part of the proof. Since the MC sample $Y^1,\ldots,Y^m$ is a Markov chain,   we will need Lemma  \ref{aux_lemma}.
(Note that behaviour of these terms would be much easier to examine if we considered
an unrealistic scenario of i.i.d.\ Monte Carlo, as in \cite{MNPR2016a}.)

We start with (\ref{Term3}). We are to show that
\begin{eqnarray}
 A_n^m &=& \frac{\sqrt{m}}{n}\sum_{i=1}^{n }\left[\grad \Zet _m (X_i,\tht)-\bar\Zet_m(\tht)\right]\\
       &=& \frac{1}{n}\sum_{i=1}^n \frac{1}{ \sqrt{m}} \sum_{k=1}^m  \left[\Psi(Y^k| X_i)   -  \bar{\Psi}(Y^k)\right] \nonumber
\end{eqnarray}
goes to 0 in probability, as $m,n\to\infty$. This a vector-valued expression, but it is enough to bound its components separately. 
Let $a_n^m$, $\psi(y|x)$ and $\bar\psi(y)$ denote any single component of $A_n^m$, $\Psi(y|x)$ and $\bar\Psi(y)$, respectively.
 Write also $\phi(y|x)=\psi(y|x)- \bar\psi(y)$. We will bound
\begin{eqnarray}
\label{form1} \Ex (a_n^m)^2
=\Ex\left[\frac{1}{n}\sum_{i=1}^n \frac{1}{ \sqrt{m}} \sum_{k=1}^m  \phi(Y^k|X_i)\right]^2, 
\end{eqnarray}
where the symbol $\Ex$ refers to the expectation with respect to the both samples (the i.i.d. variables $X_1,\ldots,X_n$ and Markov chain  $Y^1,\ldots,Y^m$ 
started at $\nu$). If we first fix $Y^k$s then the random variables $\phi_m(X_i)=\sum_{k=1}^m\phi(Y^k|X_i)/\sqrt{m}$ are i.i.d.\ and centered.
Therefore the expectation with respect to the $X_i$s in \eqref{form1}
is the variance of a mean of i.i.d.\ summands, equal to $\frac{1}{n}$ times the variance of a single summand. Consequently, 
\begin{equation}  \Ex (a_n^m)^2
\label{form2}
=\frac{1}{n}  \Ex \left[\frac{1}{ \sqrt{m}} \sum\limits_{k=1}^m  \phi(Y^k| X)\right]^2 \nonumber
\end{equation}
(expectation with respect to $X\sim g$ and the MC sample $Y^k$). Clearly, 
\begin{equation}
\label{form3}  \Ex (a_n^m)^2=
\frac{1}{nm}  \sum_{k=1}^m \Ex \phi^2(Y^k | X) + 
\frac{2}{nm}  \sum_{1\leq k<l\leq m} \Ex \phi(Y^k | X) \phi(Y^l|X). 
\end{equation}
Now, if we fix $X$ and consider randomness in $Y^k$s then $\phi(Y^k|X)$ is a functional of the Markov chain.
Since $\Ex_{Y \sim h}  \phi(Y|x)=0$ for each $x$, we are in a position to apply Lemma \ref{aux_lemma}.
Consequently,  for $k \leq l$ we have
\begin{equation}
\label{form4}
\Ex_{X \sim g,\nu} \phi(Y^k| X) \phi(Y^l|X) \leq \left\Vert \frac{q}{h}\right\Vert_\infty    
 \Ex_{X \sim g, Y \sim h} \phi^2(Y|X)  \rho ^{l-k}. \nonumber
\end{equation}
The norm and the expectation on the right-hand side of \eqref{form4} are finite, by Assumptions 1 and 5 of the Theorem. 
Since $m+2\sum_{1\leq k<l\leq m} \rho^{l-k}\leq m(1+\rho)/(1-\rho)$, from \eqref{form3} we obtain 
\begin{equation}
 \Ex (a_n^m)^2\leq \frac{1}{n} \left\Vert \frac{q}{h}\right\Vert_\infty  \Ex_{X \sim g, Y \sim h} \phi^2(Y|X) 
     \frac{1+\rho}{1-\rho}.
\end{equation}
Therefore $\Ex (a_n^m)^2\to 0$ as $m\to\infty$ and $n\to\infty$ (at an arbitrary rate). Consequently,
$a_n^m\topr 0$ and we have proved asymptotic negligibility of \eqref{Term3} (i.e.\ that this term is $\op(1/\sqrt{m})$).

The last step is bounding \eqref{Term2}. We are to show that
\begin{eqnarray}
\label{asymp121}
 B_n^m&=&\frac{\sqrt{m}}{n}\sum_{i=1}^{n }\left[\frac{\nabla \Zet _m (X_i,\tht)}
                 {\Zet _m(X_i,\tht)}- \nabla \Zet _m (X_i,\tht) \right] \nonumber \\
       &=&\frac{1}{n}\sum_{i=1}^{n }\frac{(1-\Zet _m (X_i,\tht))}
                 {\Zet _m(X_i,\tht)}\frac{1}{\sqrt{m}} \sum_{k=1}^m\Psi(Y^k|X_i)\nonumber         
\end{eqnarray}
goes to 0 in probability. As in the previous part of the proof, we will consider a single component $b_n^m$ of vector $B_n^m$.
By Cauchy-Schwarz inequality, 
\begin{equation}
|b_n^m|\leq\sqrt{\frac{1}{n} \sum_{i=1}^n \frac{[\Zet _m(X_i,\tht) - 1]^2}{\Zet ^2_m(X_i,\tht)}}\:
\sqrt{\frac{1}{n} \sum_{i=1}^n \left| \frac{1}{ \sqrt{m}} \sum_{k=1}^m
\psi(Y^k| X_i) \right|^2}.
\end{equation}
By Assumption 7(a) we  obtain that for arbitrary
$\varepsilon>0, \eta >0$ and sufficiently large $m$ with probability at least $1-\eta$
for every $x \in \mX$
\begin{equation}\nonumber
1-\varepsilon \leq \Zet _m (x,\tht) \leq 1+\varepsilon.
\end{equation}
Therefore, the term under the first square root in (\ref{asymp121}) tends in probability to 0, because with probability at least $1-\eta$
\begin{equation}\nonumber
\frac{1}{n} \sum_{i=1}^n \frac{[\Zet _m(X_i,\tht) - 1]^2}{\Zet ^2_m(X_i,\tht)} \leq \sup_{x \in \mX} \frac{[\Zet _m(x,\tht) - 1]^2}{\Zet ^2_m(x,\tht)} \leq \frac{\varepsilon^2}{(1-\varepsilon )^2} \:,
\end{equation}
if $m$ is sufficiently large.
To show that the second square root in (\ref{asymp121}) is bounded in probability we use  Markov's inequality and proceed 
similarly to bounding \eqref{form2}. We can apply Lemma 
\ref{aux_lemma}, because $\Ex_{Y \sim h} \psi (Y|x)=0$ for each $x$ and we obtain
\begin{equation}\nonumber
\Ex \frac{1}{n} \sum_{i=1}^n \left| \frac{1}{ \sqrt{m}} \sum_{k=1}^m
\psi(Y^k| X_i) \right|^2\leq \left\Vert \frac{q}{h}\right\Vert_\infty  \Ex_{X \sim g, Y \sim h} \psi^2(Y|X) 
     \frac{1+\rho}{1-\rho}.
\end{equation}
It follows that $b_n^m\topr 0$ and this ends the proof.
\end{proof}

\begin{proof}[Proof of Lemma {\ref{aux_lemma}}]
We consider Hilbert space $L^2_\pi$ of functions $\gie:\mY\to\Rl$ with finite norm $\int \gie(y)^2\pi(\d y) =\Vert \gie\Vert_\pi^2$.  
The transition kernel $P(y,\;\cdot\;)$ of Markov chain $(Y^k)_{k \geq 1}$ is associated with linear operator $P$ defined by 
$ P\gie(y)=\int_{\mY}\gie(z)P(y,\d z)$. We also define operator $\Pi$ by $\Pi\gie(y)=\int_{\mY}\gie(z)\pi(\d z)$. 
We assume that the Markov chain is reversible and geometrically ergodic that is equivalent to 
$\Vert P-\Pi\Vert_{L^2_\pi} =\rho$, where $\Vert\; \cdot\; \Vert_{L^2_\pi}$ is the operator norm and $1-\rho>0$ is the spectral gap \cite{RobertsRosenthal1997}.

Below, $\Ex_\nu$ denotes the expectation with respect to the Markov chain with the initial distribution $\nu$. We start with the following observation: 
 \begin{eqnarray*}
\Ex_\nu [\gie(Y^k)\gie(Y^l)]&=&\Ex_\nu [\mathbb{E} (\gie(Y^k)\gie(Y^l)|Y^k)] \nonumber\\
&=&\Ex_\nu [\gie(Y^k) \Ex (\gie(Y^l)|Y^k)] \nonumber\\
  &=&\Ex _\nu [\gie(Y^k)P^{l-k}\gie(Y^k)].
\end{eqnarray*}
Using the fact that $\pi$ is a stationary distribution we obtain 
\begin{eqnarray*}
\left| \mathbb{E}_\nu [\gie(Y^k)P^{l-k}\gie(Y^k)]\right|&=&\left|\int \gie(y)P^{l-k}\gie(y) \int P^k(z, \d y) \nu(\d z)\right|\\
 &=&\left|\int \gie(y)P^{l-k}\gie(y) \int P^k(z, \d y)\frac{\d\nu}{\d\pi}(z) \pi(\d z))\right|\\
 &\leq&  \left\Vert \frac{ \d \nu}{  \d \pi}\right\Vert_\infty \int |\gie(y)P^{l-k}\gie(y)|\int P^k(z, \d y)\pi( \d z)\\
 &=&\left\Vert \frac{\rm \d \nu}{\rm  \d \pi}\right\Vert_\infty  \int |\gie(y)P^{l-k}\gie(y)| \pi(\rm \d y)
 \end{eqnarray*}
Therefore, we obtain that 
\begin{equation}\nonumber
\left| \Ex_\nu [\gie(Y^k)\gie(Y^l)] \right|  \leq \left\Vert \frac{\rm \d \nu}{\rm  \d \pi}\right\Vert_\infty \Ex_\pi \left| \gie(Y)P^{l-k}\gie(Y)
\right|.
\end{equation}
Moreover, using the fact that $\Pi\gie=0,$ the Cauchy-Schwarz inequality, the definition of the operator norm and the property that 
$P^k -\Pi = (P- \Pi)^k $ for $k \geq 1$ we obtain
\begin{eqnarray*}
   \mathbb{E}_\pi |\gie(Y)P^{l-k}\gie(Y)|&=&\int|\gie(y)(P^{l-k}-\Pi)\gie(y)|\pi(\rm \d y)\\
   &\leq& \Vert \gie\Vert_\pi \, \Vert(P^{l-k}-\Pi)\gie\Vert_\pi\\
   &\leq& \Vert \gie\Vert_\pi^2 \, \Vert P^{l-k}-\Pi\Vert_{L^2_\pi}\\
   &\leq& \Vert \gie\Vert_\pi^2 \, \Vert P-\Pi\Vert_{L^2_\pi} ^{l-k}= \Vert \gie\Vert_\pi^2 \,\rho^{l-k},
 \end{eqnarray*}
which finishes the proof.
\end{proof}


\end{document}